\documentclass[envcountsect,referee]{svjour3}
\smartqed  
\usepackage{graphicx}
\usepackage{latexsym,bm,amsmath,amssymb,mathrsfs}

\usepackage{longtable,rotating}
\usepackage{multirow,booktabs,cases}
\usepackage[misc]{ifsym}
\usepackage[style=1]{mdframed}
\usepackage{algorithm,algorithmic}
\usepackage[colorlinks=true]{hyperref}
\hypersetup{urlcolor=blue,citecolor=blue,linkcolor=blue}
\usepackage{epstopdf}

\def\[{\begin{equation}}
\def\]{\end{equation}}

\def\t{\top}
\def\A{{\mathcal A}}
\def\B{{\mathcal B}}

\def\u{{\bm u}}

\numberwithin{equation}{section}
\allowdisplaybreaks

\begin{document}
\graphicspath{{./FIG/}}

\title{A class of second-order cone eigenvalue complementarity problems for higher-order tensors}

\titlerunning{A class of second-order cone eigenvalue complementarity problems}     

\author{Jiaojiao Hou \and Chen Ling \and Hongjin He}


\institute{J. Hou\and C. Ling \and H. He\at
Department of Mathematics, School of Science, Hangzhou Dianzi University, Hangzhou, 310018, China.\\
\email{houjiaojmath@163.com}
\and C. Ling \at
{macling@hdu.edu.cn}
\and H. He (\Letter) \at
\email{hehjmath@hdu.edu.cn}
 }

\date{Received: date / Accepted: date}

\maketitle

\begin{abstract}
In this paper, we consider the {\it second-order cone tensor eigenvalue complementarity problem} (SOCTEiCP) and present three different reformulations to the model under consideration. Specifically, for the general SOCTEiCP, we first show its equivalence to a particular {\it variational inequality} under reasonable conditions. A notable benefit is that such a reformulation possibly provides an efficient way for the study of properties of the problem. Then, for the symmetric and sub-symmetric SOCTEiCPs, we reformulate them as appropriate nonlinear programming problems, which are extremely beneficial for designing reliable solvers to find solutions of the considered problem. Finally, we report some preliminary numerical results to verify our theoretical results.



\end{abstract}

\section{Introduction}\label{Introd}

A tensor, as a natural extension of the concept of matrices, is a multidimensional array, whose order refers to the dimensionality of the array, or equivalently, the number of indices needed to label a component of that array. Mathematically, a real $m$-th order $n$-dimensional square tensor, denoted by $\mathcal{A}$, can be expressed as $\mathcal{A}= (a_{i_1\ldots i_m} )$, where each component $a_{i_1\ldots i_m} \in \mathbb{R}$ for $1\leq i_1, \ldots, i_m\leq n$. For the sake of convenience, we denote by $\mathcal{T}_{m,n}$ the space of $m$-th order $n$-dimensional real square tensors. If the entries of $\mathcal{A}\in \mathcal{T}_{m,n}$ are invariant under any permutation of its indices, we call $\mathcal{A}$ a {\it symmetric} tensor.
For a vector $x: = (x_1, \ldots,x_n)^\top\in \mathbb{R}^n$ and a tensor $\mathcal{A}= (a_{i_1\ldots i_m})\in  \mathcal{T}_{m,n}$,  we define $\mathcal{A}x^{m-1}$ as an $n$-dimensional vector whose $i$-th component is given by
$$
(\mathcal{A}x^{m-1})_i=\sum_{i_2,\ldots,i_m=1}^na_{ii_2\ldots i_m}x_{i_2}\cdots x_{i_m},~~{\rm for~}i=1,2,\ldots,n,$$
and throughout, let $\mathcal{A}x^m$ be the value at $x$ of a homogeneous polynomial, defined by
$$\mathcal{A}x^{m}=\sum_{i_1,i_2,\ldots,i_m=1}^na_{i_1i_2\ldots i_m}x_{i_1}x_{i_2}\cdots x_{i_m}.$$

Given two tensors $\mathcal{A},\;\mathcal{B}\in \mathcal{T}_{m,n}$, we say that $(\mathcal{A},\mathcal{B})$ is an {\it identical singular pair}, if
$$
\left\{x\in \mathbb{R}^n\backslash\{0\}~|~\mathcal{A}x^{m-1}=0,\;\mathcal{B}x^{m-1}=0\right\}\neq\emptyset.
$$
Under the assumption that $(\mathcal{A},\mathcal{B})$ is not an identical singular pair, we
call $(x,\lambda)\in (\mathbb{C}^n\backslash\{0\})\times \mathbb{C}$ an {\it eigenvector-eigenvalue pair} of $(\mathcal{A}, \mathcal{B})$, if we could find a {\it nonzero} solution $x$ to the following $n$-system of equations
\begin{equation}\label{ABEigen}
(\lambda\mathcal{A}-\mathcal{B})x^{m-1}=0,
\end{equation}
where the nonzero vector $x$ satisfying \eqref{ABEigen} is also called an eigenvector of $(\mathcal{A},\mathcal{B})$, and $\lambda$ is the associated eigenvalue to the eigenvector $x$ of $(\mathcal{A},\mathcal{B})$. The concept of eigenvector-eigenvalue pair for tensors can be dated back to the independent work of Lim \cite{Lim05} and Qi \cite{Q05}, and the appearance of such a concept has greatly initiated the rapid developments of the spectral theory of tensors. In 2009, Chang et al. \cite{CPZ09} further introduced a unified definition of eigenvector-eigenvalue pair for general square tensors, thereby making the study of tensor in the direction of complementarity problems more interesting, e.g., see \cite{HQ16,LHQ15,SQ16,LQX15}. Indeed, the importance of tensor and its eigenvalue/eigenvector has been highlighted due to the concise mathematical framework for formulating and analyzing many real-world problems in areas such as magnetic resonance imaging \cite{BV08,QYW10}, higher-order Markov chains \cite{NQZ09} and
best-rank one approximation in data analysis \cite{QWW09}. Accordingly, many nice properties such as the
Perron-Frobenius theorem for eigenvalues/eigenvectors of nonnegative square tensor have
been established, see, e.g., \cite{CPZ08,YY10}.

In the literature, e.g., see \cite{FP03,FP97}, complementarity problems have been developed well due to the widespread applications in engineering and economics. As an important special case of complementarity problems, the {\it eigenvalue complementarity problem} (EiCP) for matrices also has been studied extensively, see \cite{AR13,AS11,CS10,JRRS08,JSR07,JSRR09,VG97} for example. Most recently, the EiCP for matrices has been generalized to tensors in \cite{LHQ15}, where the authors called it {\it tensor generalized eigenvalue complementarity problem} (TGEiCP) which has been further studied from both theoretical and numerical perspective in \cite{CQ15,CYY15,FNZ16,YSXY16}. It is well known that the second-order cone is an important class of cones in applied mathematics, whose high applicability encourages the study of the specific EiCP on second-order cones for matrices, which is called {\it second-order cone eigenvalue complementarity problems} (SOCEiCP). By utilizing the special structure of second-order cones, some more interesting results have been developed, see, e.g. \cite{AR14,FFJS15}. To the best of our knowledge, the development of TGEiCP is still in its infancy. Therefore, a natural question is that can we also extend the SOCEiCP to tensors and obtain more interesting properties for such a specific TGEiCP.


In this paper, we study the {\it second-order cone tensor eigenvalue complementarity problem} (SOCTEiCP), which seeks a nonzero vector $x\in \mathbb{R}^n$ and a scalar $\lambda\in \mathbb{R}$ satisfying
\begin{equation}\label{SOCTEiCP}
x \in \mathcal{K},~~ w:=\left(\lambda \mathcal{A} -\mathcal{B}\right)x^{m-1}\in \mathcal{K}^* ~~~{\rm and}~~~\langle x, w\rangle=0,
\end{equation}
where $\mathcal{A}$ and $\mathcal{B}$ are two real $m$-th order $n$-dimensional tensors, $\langle\cdot,\cdot\rangle$ denotes the standard inner product in real Euclidean space, $\mathcal{K}^*$ is the dual cone of $\mathcal{K}$, and here $\mathcal{K}$ is the second-order cone defined by \begin{equation}\label{Soccone}
\mathcal{K}:=\mathcal{K}^{n_1}\times \mathcal{K}^{n_2}\times \cdots \times \mathcal{K}^{n_r}
\end{equation}
with
$
\mathcal{K}^{n_i}:=\left\{x^i\in \mathbb{R}^{n_i} ~|~ x_{\circ}^i\geq\|x_{\bullet}^i\|\right\}$ for $i=1,2,\ldots,r$.
Note that the $n$-dimensional vector $x$ can also be separated into $r$ parts, i.e., $x:=(x^1, x^2, \ldots , x^r)\in \mathbb{R}^{n_1}\times\mathbb{R}^{n_2}\times \cdots \times\mathbb{R}^{n_r}$ with $\sum_{i=1}^r{n_i}=n$, and each part $x^i:=(x_{\circ}^i, x_{\bullet}^i)\in \mathbb{R}\times\mathbb{R}^{n_i-1}$ for $i=1,\ldots,r$. It is obvious that each cone $\mathcal{K}^{n_i}$ is {\it pointed} and {\it self-dual}, which means that $(\mathcal{K}^{n_i})^*=\mathcal{K}^{n_i}$, where the dual cone $(\mathcal{K}^{n_i})^*$ of $\mathcal{K}^{n_i}$ is defined by
$$(\mathcal{K}^{n_i})^*:=\left\{y^i\in \mathbb{R}^{n_i}~|~\langle y^i,x^i\rangle\geq 0,~\forall ~x^i\in \mathcal{K}^{n_i}\right \}.$$
As a consequence, we know that  $\mathcal{K}$ is also pointed and self-dual.

The main contributions of this paper are summarized as follows. First, we show that SOCTEiCP \eqref{SOCTEiCP} is provably equivalent to a variational inequality, thereby establishing the existence of a solution to SOCTEiCP \eqref{SOCTEiCP}. Actually, one more important benefit is that such a characterization might provides an efficient way for the study of properties (e.g., sensitivity and stability) of SOCTEiCP \eqref{SOCTEiCP} in the context of variational inequality. Then, we focus on two special cases of SOCTEiCP \eqref{SOCTEiCP} with {\it symmetric} and {\it sub-symmetric} tensors, and reformulate both of them as two nonlinear programming problems for the purpose of designing numerical algorithms to find some of their eigenvector-eigenvalue pairs. To illustrate the solvability of SOCTEiCP \eqref{SOCTEiCP}, we employ the so-named {\it scaling-and-projection algorithm} (SPA) \cite{LHQ15} to solve SOCTEiCP \eqref{SOCTEiCP} and report some preliminary computational results to verify the reliability of SPA.

The structure of this paper is divided into five parts. In Section \ref{Reformu}, we first show that SOCTEiCP \eqref{SOCTEiCP} is essentially equivalent to a variational inequality problem. In Section \ref{spc}, we consider two special cases of SOCTEiCP \eqref{SOCTEiCP}. More concretely, in Section \ref{Secsym}, we are concerned with the symmetric SOCTEiCP, that is, the underlying tensors $\mathcal{A}$ and $\mathcal{B}$ are symmetric. Based upon such a symmetry condition, we can gainfully formulate the symmetric SOCTEiCP as a fractional polynomial optimization problem. As a more general case, in Section \ref{nonlinear pro}, we discuss the case where SOCTEiCP \eqref{SOCTEiCP} has two sub-symmetric tensors $\mathcal{A}$ and $\mathcal{B}$. Similarly, we also give a nonlinear programming formulation for the sub-symmetric SOCTEiCP. In Section \ref{num}, we report some numerical results to verify the reliability of the SPA proposed in \cite{LHQ15}. Finally, we complete this paper with drawing some concluding remarks in Section \ref{Concl}.

\medskip

\noindent{\bf Notation.} Let $\mathbb{R}^n$ denote the real Euclidean space of column vectors of length $n$. The superscript $^\top$ represents the transpose. Denote $\mathbb{R}_+^n:=\{x\in \mathbb{R}^n\;|\;x\geq 0\}$. For given $x\in \mathbb{R}^n$, we also rewrite $x:=(x_1,x_2,\ldots,x_n)^\top$ as $r$ parts, i.e., $x:=(x^1, x^2, \ldots , x^r)\in \mathbb{R}^{n_1}\times\mathbb{R}^{n_2}\times \cdots \times\mathbb{R}^{n_r}$, with $\sum_{i=1}^r{n_i}=n$ and  $x^i:=(x_{\circ}^i, x_{\bullet}^i)\in \mathbb{R}\times\mathbb{R}^{n_i-1}$ for $i=1,\ldots,r$. For $\mathcal{A}\in \mathcal{T}_{m,n}$ and a subset $J$ of the index set $[r]:=\{1,2,\ldots,r\}$, we denote by $\mathcal{A}_J$ the principal sub-tensor of $\mathcal{A}$, which is obtained by homogeneous polynomial $\mathcal{A}x^m$ for all $x=(x^1,x^2,\ldots,x^r)$ with $x^i=0$ for $i\in [r]\backslash J$. So, $\mathcal{A}_J$ is a tensor of order $m$ and dimension $|J|$, where $|J|=\sum_{i\in J}n_i$. Correspondingly, denote by $x_J$ the sub-vector of $x=(x^1, x^2, \ldots , x^r)$, which is obtained by removing the components $x^i$ with $i\in [r]\backslash J$. For given $\mathcal{C}:=(c_{i_1i_2\ldots i_m})\in \mathcal{T}_{m,n}$ and $x\in \mathbb{R}^n$, $\mathcal{C} x^{m-2}$ denotes the $n\times n$ matrix with the $(i,j)$-th element given by
$$\left(\mathcal{C} x^{m-2}\right)_{ij}:=\sum_{i_3,\ldots,i_m=1}^nc_{iji_3\ldots i_m} x_{i_3}\cdots  x_{i_m}.$$

\section{A variational inequality characterization to SOCTEiCP \eqref{SOCTEiCP}}\label{Reformu}

As a special case of TGEiCP introduced in \cite{LHQ15}, it is clear that SOCTEiCP \eqref{SOCTEiCP} also has at least one solution under some mild conditions. In this section, we reformulate SOCTEiCP \eqref{SOCTEiCP} as a variational inequality from a different perspective used in \cite{LHQ15}. We start this section with recalling some basic definitions and properties on second-order cones and tensors, which will be used in this paper.

For the second-order cone $\mathcal{K}$ defined by (\ref{Soccone}), it is well known that the complementarity condition on $\mathcal{K}$ can be decomposed
into complementarity conditions on each $\mathcal{K}^{n_i}$, that is,
\begin{equation}\label{xyii}
\begin{array}{l}
x,y\in \mathcal{K}~{\rm and~}\langle x,y\rangle=0\\
~~~~~~~~~~~~~~~~~~\Leftrightarrow~~x^i,y^i\in \mathcal{K}^{n_i}~{\rm and~}\langle x^i,y^i\rangle=0,~~~{\rm for}~i=1,2,\ldots,r,
\end{array}\end{equation}
where $x=(x^1,x^2,\ldots,x^r),~y=(y^1,y^2,\ldots,y^r)\in \mathbb{R}^{n_1}\times \mathbb{R}^{n_2}\times \cdots\times\mathbb{R}^{n_r}$.
Moreover, for any $z=(z_{\circ},z_{\bullet})$ and $w=(w_{\circ},w_{\bullet})\in \mathbb{R}\times \mathbb{R}^{l-1}$, we define the {\it Jordan product} between $z$ and $w$ as
\begin{equation}\label{Jordan}
z\circ w:=\left(\langle z,w\rangle,~w_{\circ}z_{\bullet}+z_{\circ}w_{\bullet}\right).
\end{equation}
With the above definition of Jordan product of vectors, we have the following result from \cite{FLT01}.
\begin{proposition}\label{SoccProp}
For any vectors $z,w\in\mathbb{R}^l$, we have
$$
z,w\in \mathcal{K}^l~ {\rm and}~~\langle z,w\rangle=0 ~~~~\Leftrightarrow~~~~z,w\in \mathcal{K}^l~{\rm and}~z\circ w =0_l,
$$
where $0_l$ is a zero vector in $\mathbb{R}^l$.
\end{proposition}

For given tensors $\mathcal{A}$ and $\mathcal{B}\in \mathcal{T}_{m,n}$, we define the function $F:\mathbb{R}^n\rightarrow \mathbb{R}^n$ by
\begin{equation}\label{FRR}
F(x)=\lambda(x)\mathcal{A}x^{m-1}-\mathcal{B}x^{m-1},
\end{equation}
where \begin{equation}\label{lambda}
\lambda(x)=\frac{\mathcal{B}x^m}{\mathcal{A}x^m},
\end{equation}
which is called the {\it generalized Rayleigh quotient} related to $\mathcal{A}$ and $\mathcal{B}$. Throughout this paper, we assume that $\mathcal{A}x^m\neq 0$ for any $x\in \mathcal{K}\backslash\{0\}$, which means that  $\mathcal{A}$ (or $-\mathcal{A}$) is strictly $\mathcal{K}$-positive, i.e., $\mathcal{A}x^m > 0$ (or $-\mathcal{A}x^m > 0$) for any $x\in \mathcal{K}\backslash\{0\}$. Under this assumption, it is clear that $F$ defined by \eqref{FRR} is well-defined and continuous on $\mathcal{K}_0:=\{x=(x^1,x^2,\ldots,x^r)\in \mathcal{K}~|~ e^\top x=1\}$, where $e:=(e^1,e^2,\ldots,e^r)$ with $e^i=(1,0,\ldots,0)^\top \in \mathbb{R}^{n_i}$. Obviously, $\mathcal{K}_0$ is exactly a convex compact basis of $\mathcal{K}$.

Consider the following {\it variational inequality problem} (VIP), which refers to the task of finding a vector $\bar x\in \mathcal{K}_0$ such that
\begin{equation}\label{VIP}
F(\bar x)^\top (z-\bar x)\geq 0,~~~\forall~z\in \mathcal{K}_0.
\end{equation}
In what follows, we denote \eqref{VIP} by VIP($F,\mathcal{K}_0$) for simplicity. Since $\mathcal{K}_0$ is a nonempty convex compact set, we have the following existence result on the solutions of VIP($F,\mathcal{K}_0$) (e.g., see \cite{FP03}).
\begin{proposition}\label{VIexistence}
VIP($F,\mathcal{K}_0$) has at least one solution.
\end{proposition}

Now, we establish the equivalence of \eqref{SOCTEiCP} to VIP($F,\mathcal{K}_0$), and show that \eqref{SOCTEiCP} has at least one solution.
\begin{theorem}\label{VIThex}
If $\bar x$ is a solution of VIP($F,\mathcal{K}_0$), then $(\bar x,\bar\lambda)$ is a solution of \eqref{SOCTEiCP}, where $\bar\lambda:=\lambda(\bar x)$ and $\lambda(x)$ is defined by \eqref{lambda}.
\end{theorem}

\begin{proof}
The proof is divided into two parts by distinguishing two cases of $x_{\circ}^i$.

{\bf Case I}. We first consider the case where $\bar x_{\circ}^i>0$ for $i=1,2,\ldots,r$. Since $\bar x$ is a solution of VIP($F,\mathcal{K}_0$), it immediately follows that $\bar x$ is a minimizer of the following optimization problem
\begin{equation*}
\min_{x} \;\left\{\;F(\bar x)^\top x\;\big{|}\;  x\in\mathcal{K}_0\;\right\},
\end{equation*}
which can also be rewritten as
\begin{equation*}
\begin{array}{cl}
{\min}& F(\bar x)^\top x\\
{\rm s.t.}&(x_{\circ}^i)^2-\|x_{\bullet}^i\|^2\geq0,\\
&x_{\circ}^i\geq 0,~i=1,\ldots,r,\\
&e^\top x=1.
\end{array}
\end{equation*}
Since $\bar x_{\circ}^i>0$ for every $i=1,2,\ldots,r$, the {\it linear independence constraint qualification} holds at $\bar x$, we know that $\bar x$ satisfies the KKT conditions (see \cite[Theorem 4.2.13]{BSS06}). Therefore, there exist Lagrange multipliers $\bar \beta:=(\bar \beta_1,\bar \beta_2,\ldots,\bar \beta_r)^\top$, $\bar \gamma:=(\bar \gamma_1,\bar \gamma_2,\ldots,\bar \gamma_r)^\top \in \mathbb{R}^r$ and $\bar \delta\in \mathbb{R}$ such that
\begin{equation}\label{VInlp1}
\left\{
\begin{array}{l}
F(\bar x)=2C\bar\beta+D\bar\gamma+\bar\delta e,\\
\bar\beta_i\geq 0,~(\bar x_{\circ}^i)^2-\|\bar x_{\bullet}^i\|^2\geq 0,~\bar\beta_i\left((\bar x_{\circ}^i)^2-\|\bar x_{\bullet}^i\|^2\right)=0,~i=1,\ldots ,r,\\
\bar \gamma_i\geq 0,~ \bar x_{\circ}^i\geq 0,~\bar \gamma_i \bar x_{\circ}^i=0,~i=1,\ldots ,r,\\
e^\top \bar x=1,
\end{array}
\right.
\end{equation}
where
$$
C:=\left[
\begin{array}{cccc}
\left[
\begin{array}{c}
\bar x_{\circ}^1\\
-\bar x_{\bullet}^1
\end{array}
\right]
&0_{n_1}&\cdots&0_{n_1}\\
0_{n_2}&\left[
\begin{array}{c}
\bar x_{\circ}^2\\
-\bar x_{\bullet}^2
\end{array}
\right]
&\cdots&0_{n_2}\\
\vdots&\vdots&\ddots&\vdots\\
0_{n_r}&0_{n_r}&\cdots&\left[
\begin{array}{c}
\bar x_{\circ}^r\\
-\bar x_{\bullet}^r
\end{array}
\right]
\\
\end{array}
\right]~~~{\rm and}~~~D:=\left[\begin{array}{cccc}
e^1&0_{n_1}&\cdots&0_{n_1}\\
0_{n_2}&e^2&\cdots&0_{n_2}\\
\vdots&\vdots&\ddots&\vdots\\
0_{n_r}&0_{n_r}&\cdots&e^r
\end{array}\right]\in \mathbb{R}^{n\times r}
$$
with $0_{n_i}$ being the zero vector in $\mathbb{R}^{n_i}$ for $i=1,\ldots,r$.
By (\ref{VInlp1}), we know $\bar x^\top F(\bar x)=\bar\delta$, which implies, together with the fact that $\bar x^\top F(\bar x)=0$, that $\bar\delta=0$.
Consequently, from the first expression of (\ref{VInlp1}), we get
\begin{equation}\label{VIBBYt}
F(\bar x)=2C\bar\beta+D\bar\gamma.
\end{equation}
For the notational convenience, let us write
$$F(\bar x)=:\bar y:=(\bar y^1,\bar y^2,\ldots,\bar y^r) \in \mathbb{R}^{n_1}\times\mathbb{R}^{n_2}\times \cdots \times\mathbb{R}^{n_r}$$
with $\bar y^i=(\bar y_{\circ}^i,\bar y_{\bullet}^i)\in \mathbb{R}\times\mathbb{R}^{n_i-1}$. By (\ref{VIBBYt}), it is easy to verify that
$$\bar y_{\circ}^i-\|\bar y_{\bullet}^i\|=2\bar\beta_i\left(\bar x_{\circ}^i-\|\bar x_{\bullet}^i\|\right)+\bar\gamma_i\geq0,\;i=1,\ldots ,r,$$
which means that $\bar y^i\in \mathcal{K}^{n_i}$ for $i=1,\ldots ,r$, and hence $\bar y \in \mathcal{K}~(=\mathcal{K}^*)$. Consequently, we have $\bar\lambda\mathcal{A}\bar x^{m-1}-\mathcal{B}\bar x^{m-1} \in \mathcal{K}$ as well as $\bar x\in \mathcal{K}$. Since $\bar x^\top F(\bar x)=0$, we know $\bar x^\top\left(\bar\lambda\mathcal{A}\bar x^{m-1}-\mathcal{B}\bar x^{m-1}\right)=0$, which implies that $(\bar x,\bar\lambda)$ is a solution of \eqref{SOCTEiCP} because of $\bar x\neq 0$.

{\bf Case II}. Now we consider the case of $\bar x_{\circ}^j=0$ for some $j$. It follows from $\bar x^j\in\mathcal{K}^{n_j}$ that $\bar x^j=0$. Using the constraint $e^\top \bar x=1$, $r\geq 2$, we assume, for simplicity, that there is exactly one such $j$ and that $j=1$, i.e., $\bar x = (\bar x^1,\bar u)\in \mathbb{R}^{n_1}\times \mathbb{R}^{n-n_1}$ with
$$
\bar x^1=0,~~\bar u=(\bar x^2,\ldots,\bar x^r)~~{\rm and}~~\bar x^i\neq0,~i=2,\ldots,r.
$$
Correspondingly, we have
$$
F(\bar x)=\lambda(\bar u)\left[
\begin{array}{c}
\mathcal{A}_{12}\bar u^{m-1}\\
\mathcal{A}_{22}\bar u^{m-1}
\end{array}
\right]-\left[
\begin{array}{c}
\mathcal{B}_{12}\bar u^{m-1}\\
\mathcal{B}_{22}\bar u^{m-1}
\end{array}
\right]=\left[
\begin{array}{c}
F_1(\bar u)\\
F_2(\bar u)
\end{array}
\right]
\quad {\rm with}\;\;\lambda(\bar u):=\frac{\mathcal{B}_{22}\bar u^{m}}{\mathcal{A}_{22}\bar u^{m}}.$$
Here, for a given tensor $\mathcal{C}:=(c_{i_1i_2\ldots i_m})\in \mathcal{T}_{m,n}$, let $\mathcal{C}_{12}$ and $\mathcal{C}_{22}$ be sub-tensors of $\mathcal{C}$, whose elements are defined by
$$(\mathcal{C}_{12})_{i_1i_2\ldots i_m}:=c_{i_1(n_1+i_2)\ldots (n_1+i_m)},~~i_1=1,2,\ldots, n_1;~i_2,\ldots,i_m=1,2,\ldots, n-n_1,$$
and $$(\mathcal{C}_{22})_{i_1i_2\ldots i_m}:=c_{(n_1+i_1)(n_1+i_2)\ldots (n_1+i_m)},~~i_1,i_2,\ldots,i_m=1,2,\ldots, n-n_1,$$
respectively. Since $\bar x = (0,\bar u)$ is a solution of VIP($F,\mathcal{K}_0$), it turns out that
\begin{equation}\label{VI10}
F(\bar x)^\top (x-\bar x)=F_1(\bar u)^\top x^1+F_2(\bar u)^\top (u-\bar u)\geq 0
\end{equation}
for any $x^1\in \mathbb{R}^{n_1}$ and $u\in \mathbb{R}^{n-n_1}$ satisfying $x:=(x^1,u)\in \mathcal{K}_0$. Taking $x^1=0$ in \eqref{VI10} immediately leads to
$$
F_2(\bar u)^\top (u-\bar u)\geq 0, ~~~\forall ~u\in \bar{\mathcal{K}}_0,
$$ where $\bar {\mathcal{K}}_0:=\{u=(u^2,\ldots,u^r)\in \mathbb{R}^{n-n_1}~|~u_{\circ}^i\geq\|u_{\bullet}^i\|, u_{\circ}^2+\ldots +u_{\circ}^r=1\}$, which means that $\bar u$ is a solution of VIP($F_2,\bar {\mathcal{K}}_0$). Since $\bar x^i\neq0$ for $i=2,\ldots,r$, it follows the proof of {\bf Case I} that $(\bar u,\lambda(\bar u))$ is a solution to
\begin{equation}\label{SOCTEiCP1}
u \in \bar{\mathcal{K}},~~ v:=\lambda \mathcal{A}_{22} u^{m-1}-\mathcal{B}_{22}u^{m-1}\in \bar{\mathcal{K}} ~~~{\rm and}~~~\langle u, v\rangle=0,
\end{equation}
where $\bar{\mathcal{K}}$ is the second-order cone defined by $\bar{\mathcal{K}}:=\mathcal{K}^{n_2}\times \cdots \times \mathcal{K}^{n_r}$. By (\ref{SOCTEiCP1}), we have $F_2(\bar u)^\top \bar u=0$. Consequently, by taking $u=0$ in \eqref{VI10}, it can be seen that $F_1(\bar u)^\top x^1\geq 0$ for any $x^1\in \mathcal{K}^{n_1}$, and hence $F_1(\bar u)\in (\mathcal{K}^{n_1})^*=\mathcal{K}^{n_1}$. Therefore, we conclude that $(\bar x,\lambda(\bar x))$ is a solution of (\ref{SOCTEiCP}). We complete the proof. \qed
\end{proof}

As a direct result of Proposition \ref{VIexistence} and Theorem \ref{VIThex}, if $\mathcal{A}$ and $\mathcal{B}$ are matrices, we can easily obtain the solution existence result of SOCEiCPs.


\section{Nonlinear programming for SOCTEiCP \eqref{SOCTEiCP} with special structure}\label{spc}

In this section, we focus on two special cases of SOCTEiCP \eqref{SOCTEiCP} with {\it symmetric} and {\it sub-symmetric} tensors $\mathcal{A}$ and $\mathcal{B}$, and reformulate them as nonlinear programming problems for the purpose of utilizing or designing optimization methods to find solutions of the model.

\subsection{The symmetric SOCTEiCP}\label{Secsym}
When $\mathcal{A}$ and $\mathcal{B}$ are symmetric,
it is easy to see that the gradient of the generalized Rayleigh quotient
$\lambda(x)$ defined in (\ref{lambda}) is
\begin{equation}\label{lambdaGrad}
\nabla \lambda(x)=\frac{m}{\mathcal{A}x^m}\left(\mathcal{B}x^{m-1}-\lambda(x)\mathcal{A}x^{m-1}\right).
\end{equation}
Here we should notice that the gradient formula \eqref{lambdaGrad}
of the Rayleigh quotient holds only for the case where $\mathcal{A}$ and $\mathcal{B}$ are both symmetric.

The following lemma states two fundamental properties of $\lambda(x)$, which have been studied in \cite{QJH04} for matrices. Its proof is elementary and skipped here.

\begin{lemma}\label{ECPLemma3}
For all $x\in \mathbb{R}^n\backslash \{0\}$, the following statements hold:
\begin{itemize}
\itemindent 8pt
\item[\rm(i)] $\lambda(\tau x)=\lambda(x)$, ~~$\forall ~\tau >0$;
\item[\rm(ii)] $x^\top \nabla\lambda(x)=0$.
\end{itemize}
\end{lemma}


Now, we consider the following fractional programming model:
\begin{equation*}
\max_{x}\;\left\{\;\lambda(x):=\frac{\mathcal{B}x^m}{\mathcal{A}x^m}\;\big{|}\; x\in\mathcal{K}_0\;\right\},
\end{equation*}
which, from the definition of $\mathcal{K}_0$, can also be rewritten as
\begin{equation}\label{NLP1}
\begin{array}{cl}
{\rm max}& \lambda(x):=\frac{\mathcal{B}x^m}{\mathcal{A}x^m}\\
{\rm s.t.}&(x_{\circ}^i)^2-\|x_{\bullet}^i\|^2\geq0,\\
&x_{\circ}^i\geq 0,~i=1,\ldots,r,\\
&e^\top x=1.
\end{array}
\end{equation}
Then, as a result of Theorem \ref{VIThex}, the following theorem clarifies the relationship between \eqref{SOCTEiCP} and \eqref{NLP1}, that is, solving the symmetric SOCTEiCP actually reduces to finding a stationary point of \eqref{NLP1}, which is greatly helpful for efficiently solving the model under consideration.

\begin{theorem}\label{ThforTEiCPOpt} Assume that $\mathcal{A}$ and $\mathcal{B}$ are symmetric tensors and $\mathcal{A}$ is strictly $\mathcal{K}$-positive. Let $\bar x$ be a stationary point of \eqref{NLP1}. Then $(\bar x,\bar\lambda)$ is a solution of \eqref{SOCTEiCP}, where $\bar\lambda:=\lambda(\bar x)$ and $\lambda(x)$ is defined by \eqref{lambda}.
\end{theorem}

For given nonempty subset $J\subset[r]$, we now consider the following second-order cone optimization problem
\begin{equation}\label{NLP2}
\begin{array}{cl}
{\rm max}& \displaystyle\lambda_J(x_J):=\frac{\mathcal{B}_J(x_J)^m}{\mathcal{A}_J(x_J)^m}\\
{\rm s.t.}& \displaystyle\sum_{i\in J}x_{\circ}^i=1,\\
&(x_{\circ}^i)^2-\|x_{\bullet}^i\|^2\geq0,\;i\in J,\\
&x_{\circ}^i\geq 0,~i\in J.
\end{array}
\end{equation}

\begin{theorem}\label{ThforTEiCPOpt} Assume that $\mathcal{A}$ and $\mathcal{B}$ are symmetric tensors and $\mathcal{A}$ is strictly $\mathcal{K}$-positive.
 Let $(\bar x, \bar \lambda)$ be a solution of SOCTEiCP \eqref{SOCTEiCP} with the second-order cone $\mathcal{K}$ defined by \eqref{Soccone}. There exists an nonempty subset
  $J\subset [r]$ such that $\bar x_J$ is a stationary point of \eqref{NLP2}.
\end{theorem}

\begin{proof}
By the homogeneity of the complementarity system SOCTEiCP with respect to $x$, we assume, without loss of generality,
that $\bar x\in \mathcal{K}\backslash \left\{0\right\}$ satisfies $e^\top\bar x=1$. Let us write
$\bar x=(\bar x^1,\bar x^2,\ldots,\bar x^r)\in \mathbb{R}^{n_1}\times\mathbb{R}^{n_2}\times \cdots \times\mathbb{R}^{n_r}$
with $\bar x^i=(\bar x_{\circ}^i,\bar x_{\bullet}^i) \in\mathbb{R}\times\mathbb{R}^{n_i-1}$ for $i=1,2,\ldots,r$. For the sake of simplicity, we write
$$\bar y:=\frac{m}{\mathcal{A}\bar x^m}\left(\bar\lambda\mathcal{A}\bar x^{m-1}-\mathcal{B}\bar x^{m-1}\right).$$
It is clear that $\bar y\in \mathcal{K}$ as $\bar w=\bar\lambda\mathcal{A}\bar x^{m-1}-\mathcal{B}\bar x^{m-1}\in \mathcal{K}$ and $\mathcal{A}\bar x^m>0$.
Since $(\bar x,\bar\lambda)$ is a solution of SOCTEiCP, which implies $\langle\bar x,\bar y\rangle=0$. By (\ref{Jordan}) and Proposition \ref{SoccProp}, it holds that
\begin{equation}\label{Jordan1}
{\bar y_{\circ}^i}{\bar x_{\bullet}^i}+{\bar x_{\circ}^i}{\bar y_{\bullet}^i}=0, ~~~\forall ~i=1,2,\ldots,r.
\end{equation}
Since $\bar x\in \mathcal{K}\backslash\{0\}$, there exists an nonempty subset $J=\left\{i\in [r]~|~ \bar x_{\circ}^i>0\right\}$ of $[r]$. It is clear that $\bar x^i=0$ for $i\in[r]\backslash J$, and hence $\bar \lambda=\mathcal{ B}_J(\bar x_J)^m/\mathcal{A}_J(\bar x_J)^m$. Like Theorem \ref{VIThex}, let $\bar\delta$, $\bar\beta$, and $\bar\gamma$ be Lagrange multipliers associated to the constraints of \eqref{NLP2} respectively. Accordingly, we take $\bar\delta=0$. And for every $i\in J$, since $\bar x_{\circ}^i> 0$, we take $\bar \gamma_i=0$ and
\begin{equation}\label{beta}
\bar \beta_i=\frac{1}{2\bar x_{\circ}^i}\bar y_{\circ}^i.
\end{equation}
Obviously, $\bar \beta_i\geq 0$ and $\bar \gamma_i\bar x_{\circ}^i=0$ for every $i\in J$. Moreover, from \eqref{Jordan1} and \eqref{beta}, it is not difficult to see that
\begin{equation}\label{beta y}
-2\bar x_{\bullet}^i\bar \beta_i=\bar y_{\bullet}^i, ~~~~\forall~i\in J.
\end{equation}
Combining \eqref{beta} and \eqref{beta y} leads to
\begin{equation}\label{xybbt}
\bar y^i=2\left[
\begin{array}{c}
\bar x_{\circ}^i\\
-\bar x_{\bullet}^i
\end{array}
\right]\bar \beta_i,~~~ {\rm for}~i\in J.
\end{equation}
Moreover, it follows from $\langle\bar x,\bar y\rangle=0$ and \eqref{xyii} that $\langle\bar x^i,\bar y^i\rangle=0$ for $i=1,2,\ldots,r$. Consequently, from \eqref{xybbt}, we immediately obtain
\begin{equation}\label{xybeta}
\bar\beta_i\left((\bar x_{\circ}^i)^2-\|\bar x_{\bullet}^i\|^2\right)=\frac{1}{2}\langle \bar x^i,\bar y^i\rangle=0,\;\;\forall i\in J.
\end{equation}
Finally, using \eqref{xybbt} and \eqref{xybeta}, together with the fact that $e^\top \bar x=1$, we have
\begin{equation}\label{nlp2}
\left\{
\begin{array}{l}
\displaystyle\frac{m}{\mathcal{A}_J{(\bar x_J)}^m}[{\bar\lambda}\mathcal{A}_J{(\bar x_J)}^{m-1}-\mathcal{B}_J{(\bar x_J)}^{m-1}]=2\sum_{i\in J}(c_i)_J{\bar\beta_i}+\sum_{i\in J}(d_i)_J{\bar\gamma_i}+\bar\delta e_{J},\\
\displaystyle\sum_{i\in J}\bar x_{\circ}^i=1,\\
\bar\beta_i\geq 0,~(\bar x_{\circ}^i)^2-\|\bar x_{\bullet}^i\|^2\geq 0,~
\bar\beta_i\left((\bar x_{\circ}^i)^2-\|\bar x_{\bullet}^i\|^2\right)=0, ~i\in J,\\
\bar \gamma_i\geq 0, ~\bar x_{\circ}^i\geq 0, ~\bar \gamma_i\bar x_{\circ}^i=0,~i\in J,\\
\end{array}
\right.
\end{equation}
where $c_i$ and $d_i$ are the $i$-th columns of $C$ and $D$, respectively. By (\ref{nlp2}), we conclude that $\bar x_J$ is a stationary point of \eqref{NLP2}.
 \qed
\end{proof}

As an interesting by-product of Theorem \ref{ThforTEiCPOpt}, we have the following result showing that some special solutions of the symmetric SOCTEiCP are precisely stationary points of \eqref{NLP1}.
\begin{corollary}\label{corollary1}
 Assume that $\mathcal{A}$ and $\mathcal{B}$ are symmetric tensors and $\mathcal{A}$ is strictly $\mathcal{K}$-positive.
 Let $(\bar x, \bar \lambda)$ be a solution of SOCTEiCP with the second-order cone $\mathcal{K}$ defined by \eqref{Soccone}. If $\bar x^i_{\circ}>0$ for $i=1,2,\ldots, r$, then $\bar x$ is a stationary point of \eqref{NLP1}.
\end{corollary}

\subsection{The sub-symmetric SOCTEiCP}\label{nonlinear pro}
For many real-world problems, the symmetry assumption on the two tensors $\mathcal{A}$ and $\mathcal{B}$ is usually regarded as a little stronger condition. In this section, we consider a slightly general case where the underlying tensors $\mathcal{A}$ and $\mathcal{B}$ of SOCTEiCP \eqref{SOCTEiCP} are {\it sub-symmetric}.

Before our discussion, we first introduce the key concept of {\it sub-symmetry} on tensors. Let $\mathcal{A}\in \mathcal{T}_{m,n}$. We say that $\mathcal{A}$ is {\it sub-symmetric} with respect to the indices $\{i_2,\ldots,i_m\}$, if $\mathcal{A}_i:=(a_{ii_2\ldots i_m})_{1\leq i_2,\ldots,i_m\leq n}$, an $(m-1)$-th order $n$-dimensional higher tensor, is symmetric for every $i=1,\ldots,n$. Apparently, a symmetric tensor $\mathcal{A}$ must be sub-symmetric, but the reverse is not true. Hereafter, we further assume throughout this section that $m$ is even. Then, following the idea used in \cite{JSR07}, we introduce an additional vector $y=\lambda^{\frac {1}{m-1}}x$, i.e.,
$$y^i=\lambda^{\frac {1}{m-1}}x^i, ~~~{\rm for}~i=1,2,\ldots, r,$$
to derive the nonlinear programming formulation of the sub-symmetric SOCTEiCP, where the complementarity requirement of \eqref{SOCTEiCP} is absorbed into the objective function. More concretely, the nonlinear programming model can be expressed as follows:
\begin{equation}\label{asymmetric NLP}
\begin{array}{cl}
{\rm min}&f(x,y,w,\lambda):=\|y-\lambda^{\frac {1}{m-1}}x\|^2+(x^\top w)^2\\
{\rm s.t.}&w-\mathcal{A}y^{m-1}+\mathcal{B}x^{m-1}=0,\\
          & (x_{\circ}^i)^2-\|x_{\bullet}^i\|^2\geq0,~x_{\circ}^i\geq0,~i=1,\ldots,r,\\
          & (w_{\circ}^i)^2-\|w_{\bullet}^i\|^2\geq0,~w_{\circ}^i\geq0,~i=1,\ldots,r,\\
          &e^\top x=1,\\
          &e^\top y=\lambda^{\frac {1}{m-1}},
\end{array}
\end{equation}
where $x=(x^1,x^2,\ldots,x^r)$, $y=(y^1,y^2,\ldots,y^r)$ and $ w=(w^1,w^2,\ldots,w^r)\in \mathbb{R}^{n_1}\times\mathbb{R}^{n_2}\times \cdots \times\mathbb{R}^{n_r}$ with $x^i=(x_{\circ}^i,x_{\bullet}^i)$, $w^i=(w_{\circ}^i,w_{\bullet}^i)$ and $ y^i=(y_{\circ}^i,y_{\bullet}^i)\in \mathbb{R}\times \mathbb {R}^{n_i-1}$ for $i=1,\ldots,r$.

With the preparation of the nonlinear programming \eqref{asymmetric NLP}, we have the following theorem.
\begin{theorem}\label{asymmetric NLP1}
The sub-symmetric SOCTEiCP has a solution if and only if the nonlinear programming problem \eqref{asymmetric NLP} has a global minimum with its objective value being zero.
\end{theorem}

\begin{proof}
Let $(\bar x,\bar y,\bar w,\bar\lambda)$ be a global minimum of \eqref{asymmetric NLP} such that $f(\bar x,\bar y,\bar w,\bar\lambda)=0$. It is obvious that $\bar x,\bar w\in \mathcal{K}$ and $\bar x\neq0$. Moreover, it follows from $f(\bar x,\bar y,\bar w,\bar\lambda)=0$ that $\bar y=\bar\lambda^{\frac{1}{m-1}}\bar x$ and $\bar x^\top \bar w=0$. Consequently, it holds that
$$\bar w=\mathcal{A}\bar y^{m-1}-\mathcal{B}\bar x^{m-1}=\bar \lambda\mathcal{A}\bar x^{m-1}-\mathcal{B}\bar x^{m-1},$$
which, together with the fact that $\bar x^\top \bar w=0$, implies that $(\bar x,\bar \lambda)$ is a solution of the sub-symmetric SOCTEiCP.

Conversely, let $(\bar x,\bar \lambda)$ be a solution of the sub-symmetric SOCTEiCP. Denote $\tilde{x}:=\bar x/(e^\top \bar x)$, $\tilde{y}:=\bar \lambda^{\frac{1}{m-1}}\tilde x$, and $\tilde w:=\mathcal{A}\tilde y^{m-1}-\mathcal{B}\tilde x^{m-1}$. It is easy to verify that $(\tilde x,\tilde y,\tilde w,\tilde\lambda)$ is a global minimum of (\ref{asymmetric NLP}) satisfying $f(\tilde x,\tilde y,\tilde w,\tilde\lambda)=0$.\qed
\end{proof}

Notice that any global minimum of a nonlinear programming problem is a stationary point. Comparatively speaking, computing a stationary point is much easier than finding a global minimum. Therefore, it is important to investigate when a stationary point of nonlinear programming problem is a solution of the sub-symmetric SOCTEiCP, which will be addressed in the subsequent theorem.

\begin{theorem}\label{asymmetric NLP2}
Assume that $\mathcal{A,B}\in \mathcal{T}_{m,n}$ are sub-symmetric tensors. Let $(\bar x,\bar y,\bar w,\bar\lambda)$ be a stationary point of \eqref{asymmetric NLP} with $\bar \lambda\neq0$. Then $f(\bar x,\bar y,\bar w,\bar\lambda)=0$  if and only if $\bar\delta=\bar\eta=0$, where $\bar\delta$ and $\bar\eta$ are the Lagrange multipliers associated with the constraints $e^\top x=1$ and $e^\top y=\lambda^{\frac{1}{m-1}}$ in \eqref{asymmetric NLP}, respectively.
\end{theorem}

\begin{proof}
Since $(\bar x,\bar y,\bar w,\bar\lambda)$ is a stationary point of (\ref{asymmetric NLP}), there exist Lagrange multipliers $\bar\alpha\in \mathbb {R}^n$, $\bar\beta\in \mathbb {R}^r$, $\bar\gamma\in \mathbb {R}^r$, $\bar\mu\in \mathbb {R}^r$, $\bar\theta\in \mathbb {R}^r$, $\bar\delta\in \mathbb {R}$ and $\bar\eta\in \mathbb {R}$, such that the following KKT conditions for (\ref{asymmetric NLP}) holds
\begin{equation}\label{KKT}
\left\{
\begin{array}{l}
-2\bar\lambda^{\frac {1}{m-1}}\left(\bar y-\bar\lambda^{\frac {1}{m-1}}\bar x\right)+2\left(\bar x^\top \bar w\right)\bar w=(m-1)\left(\mathcal{B}\bar x^{m-2}\right)^\top\bar\alpha+2{\widehat C}\bar\beta+{\widehat E}\bar\gamma+\bar\delta e,\\
 2\left(\bar y-\bar \lambda^{\frac {1}{m-1}}\bar x\right)=-(m-1)\left(\mathcal{A}\bar y^{m-2}\right)^\top\bar\alpha+\bar\eta e,\\
2\left(\bar x^\top \bar w\right)\bar x=\bar\alpha+2{\widehat D}\bar\mu+{\widehat E}\bar\theta,\\
-{\frac {1}{m-1}}\bar\lambda^{{\frac {1}{m-1}}-1}\bar x^\top\left(\bar y-\bar\lambda^{\frac {1}{m-1}}\bar x\right)=-{\frac {1}{m-1}}\bar\lambda^{{\frac {1}{m-1}}-1}\bar\eta,\\
\bar w-\mathcal{A}\bar y^{m-1}+\mathcal{B}\bar x^{m-1}=0,\\
\bar\beta_i\geq0,~(\bar x_{\circ}^i)^2-\|\bar x_{\bullet}^i\|^2\geq0,~\bar\beta_i\left[(\bar x_{\circ}^i)^2-\|\bar x_{\bullet}^i\|^2\right]=0,~i=1,\ldots ,r,\\
\bar \gamma_i\geq0,~\bar x_{\circ}^i\geq0,~\bar x_{\circ}^i\bar \gamma_i=0,~~i=1,\ldots ,r,\\
\bar\mu_i\geq0,~(\bar w_{\circ}^i)^2-\|\bar w_{\bullet}^i\|^2\geq0,~\bar \mu_i\left[(\bar w_{\circ}^i)^2-\|\bar w_{\bullet}^i\|^2\right]=0,~i=1,\ldots ,r,\\
\bar\theta_i\geq0,~\bar w_{\circ}^i\geq0,~\bar w_{\circ}^i\bar\theta_i=0, ~i=1,\ldots ,r,\\
e^\top\bar x=1,\\
e^\top\bar y=\bar\lambda^{\frac {1}{m-1}},
\end{array}
\right.
\end{equation}
where $\bar\beta_i,\bar\gamma_i,\bar\mu_i,\bar\theta_i$ are the $i$-th components of vectors $\bar\beta,\bar\gamma,\bar\mu,\bar\theta\in \mathbb{R}^r$, respectively;
$$
{\widehat C}:=\left[
\begin{array}{cccc}
\left[
\begin{array}{c}
\bar x_{\circ}^1\\
-\bar x_{\bullet}^1
\end{array}
\right]
&0_{n_1}&\cdots&0_{n_1}\\
0_{n_2}&\left[
\begin{array}{c}
\bar x_{\circ}^2\\
-\bar x_{\bullet}^2
\end{array}
\right]
&\cdots&0_{n_2}\\
\vdots&\vdots&\ddots&\vdots\\
0_{n_r}&0_{n_r}&\cdots&\left[
\begin{array}{c}
\bar x_{\circ}^r\\
-\bar x_{\bullet}^r
\end{array}
\right]
\\
\end{array}
\right],~{\widehat D}:=\left[
\begin{array}{cccc}
\left[
\begin{array}{c}
\bar w_{\circ}^1\\
-\bar w_{\bullet}^1
\end{array}
\right]
&0_{n_1}&\cdots&0_{n_1}\\
0_{n_2}&\left[
\begin{array}{c}
\bar w_{\circ}^2\\
-\bar w_{\bullet}^2
\end{array}
\right]
&\cdots&0_{n_2}\\
\vdots&\vdots&\ddots&\vdots\\
0_{n_r}&0_{n_r}&\cdots&\left[
\begin{array}{c}
\bar w_{\circ}^r\\
-\bar w_{\bullet}^r
\end{array}
\right]
\\
\end{array}
\right]\in \mathbb{R}^{n\times r}$$
and ${\widehat E}:=D$ used in \eqref{VInlp1}.

Multiplying the first three expressions in (\ref{KKT}) by $\bar x^\top$, $\bar y^\top$ and $\bar w^\top$ respectively, and using the last six expressions in \eqref{KKT}, we have
$$\left\{
\begin{array}{l}
-2\bar\lambda^{\frac {1}{m-1}}\bar x^\top\left(\bar y-\bar\lambda^{\frac {1}{m-1}}\bar x\right)+2\left(\bar x^\top \bar w\right)^2=(m-1)\left(\mathcal{B}\bar x^{m-1}\right)^\top\bar\alpha+\bar \delta,\\
 2\bar y^\top\left(\bar y-\bar \lambda^{\frac {1}{m-1}}\bar x\right)=-(m-1)\left(\mathcal{A}\bar y^{m-1}\right)^\top\bar\alpha+\bar\lambda^{\frac{1}{m-1}}\bar\eta,\\
 2\left(\bar x^\top \bar w\right)^2=\bar w^\top\bar\alpha,
\end{array}
\right.
$$
which, together with the fifth expression in \eqref{KKT}, implies that
\begin{equation}\label{KKT11}
2(m-1)(\bar x^\top \bar w)^2+2f(\bar x, \bar w, \bar \lambda, \bar y)=\bar \delta+\bar \lambda^{\frac{1}{m-1}}\bar \eta,
\end{equation}
where $m\geq2$.

If $\bar \delta=\bar \eta=0$, it is clear from \eqref{KKT11} that $f(\bar x,\bar w,\bar\lambda,\bar y)=0$.
Conversely, if $f(\bar x,\bar y,\bar w,\bar \lambda)=0$, then it holds that $\bar y=\bar \lambda^{\frac{1}{m-1}}\bar x$ and $\bar x^\top\bar w=0$. By the fourth expression in (\ref{KKT}), it holds that $\bar\lambda^{\frac{1}{m-1}}\bar\eta=0$, which implies $\bar \eta=0$ from the given condition that $\bar \lambda \neq0$. Consequently, from $f(\bar x,\bar w,\bar \lambda,\bar y)=0$ and (\ref{KKT11}), we conclude that $\bar \delta=0$.
\qed
\end{proof}

\section{Numerical experiments}\label{num}

In \cite{LHQ15}, the authors introduced a so-called SPA for TGEiCP. As remarked in that paper, such an algorithm is computationally efficient as long as the underlying projection step has closed-form solution. Thus, in this section, we further report some preliminary results to verify the efficiency of SPA for solving SOCTEiCPs.

Note that the underlying SOCTEiCP has more complicated structure than the TGEiCP studied in \cite{LHQ15}. It is necessary to summarize some numerical notes on the second-order cone before the employment of SPA. For any vector $z:=(z_{\circ}, z_{\bullet})\in \mathbb{R}\times \mathbb{R}^{l-1}$, it is well known from \cite{AG03} (see also \cite{CSS03,FLT01}) that the spectral factorization of $z$ is defined as
\begin{equation}\label{xxsoc}
z=\zeta_1\u_1+\zeta_2\u_2,
\end{equation}
where $\zeta_i\in \mathbb{R}$ and $\u_i\in \mathbb{R}^l$ ($i=1,2$) are the spectral values and the associated
spectral vectors, respectively, given by
\begin{equation}\label{xxsoc2}
\zeta_i := z_{\circ} + (-1)^i\|z_{\bullet}\|
\end{equation}
and
\begin{equation}\label{xxsoc3}
\u_i:=\left\{
\begin{array}{ll}
\displaystyle\frac{1}{2}\left(1,(-1)^i\frac{z_{\bullet}}{\|z_{\bullet}\|}\right),&\;\;{\rm if~}z_{\bullet}\neq0,\\
\displaystyle\frac{1}{2}\left(1,(-1)^iw\right),&\;\;{\rm otherwise},\\
\end{array} \right.
\end{equation}
with any vector $w\in\mathbb{R}^{l-1}$ satisfying $\|w\|=1$. Clearly, decomposition \eqref{xxsoc} is unique for the case $z_{\bullet}\neq 0$. Define the projection of a given vector $z\in{\mathbb{R}^l}$ onto a convex set $\Omega$ as
\begin{equation*}
\Pi_{\Omega} (z) := \arg\min\;\left\{\|z'-z\|\;|\;z'\in\Omega\right\}.
\end{equation*}
Then, the projection of $z\in\mathbb{R}^l$ onto the second-order cone $\mathcal{K}^l$ can be further written explicitly as
\begin{equation}\label{Progmin}
\Pi_{\mathcal{K}^l} (z) := \max\left\{0,\zeta_1\right\}\u_1+\max\left\{0,\zeta_2\right\}\u_2,
\end{equation}
where $\zeta_i$ and $\u_i$ ($i=1,2$) are defined by (\ref{xxsoc2}) and (\ref{xxsoc3}), respectively. We refer the reader to \cite{FLT01} for the detailed derivation of \eqref{Progmin}.
%

Taking a close look at the SPA (see Algorithm 1 in \cite{LHQ15}), there is a notable relaxation factor $\alpha$ in the projection scheme, which was taken as $\alpha=1$ in \cite{CS10}. In fact, such a constant $\alpha$ actually plays an important role in enlarging the step size $s_k$ to achieve the acceleration of SPA in practice (see the numerical results reported in \cite{LHQ15}). Here, we can take $\alpha\in(1,8)$ empirically to speed up the convergence.

Throughout the experiments, we wrote the code in {\sc Matlab} R2012b and conducted the numerical experiments on a TOSHIBA notebook with Intel(R) Core(TM) i7-5600U CPU 2.60GHz and 8GB RAM running on Windows 7 Home Premium operating system.

In our experiments, we consider two concrete examples, where the underlying tensors $\mathcal{A}$ and $\mathcal{B}$ are symmetric. Note that $\mathcal{A}$ is a strictly $\mathcal{K}$-positive tensor. We thus take $\mathcal{A}$ as a sparse tensor throughout this section so that we can ensure and verify the strict $\mathcal{K}$-positiveness of $\mathcal{A}$. Note that all tensors here are symmetric, we only list the nonzero upper triangular entries.

\begin{example}\label{exam1}
We consider two $4$-th order $4$-dimensional symmetric tensors $\A$ and $\B$, where the second-order cone is specified as ${\mathcal K}:={\mathcal K}^2\times {\mathcal K}^2$. The tensors $\A$ and $\B$ take their components as listed in Tables \ref{tensorA1} and \ref{tensorB1}, respectively.
\end{example}

\begin{table}[!htbp]
\begin{center}
\caption{Nonzero components of the symmetric tensor ${\mathcal A}$ for Example \ref{exam1}.}\vskip 0.2mm
\label{tensorA1}
\def\temptablewidth{1\textwidth}
\begin{tabular*}{\temptablewidth}{@{\extracolsep{\fill}}cccccc}\toprule
$a_{1111}=2,$ & $a_{1311}=3$, & $a_{2211}=3$, & $a_{3311}=3$, & $a_{4411}=2$, & $a_{1122}=2,$ \\
$a_{1322}=2,$ & $a_{2222}=3$, & $a_{3322}=3$, & $a_{4422}=3$, & $a_{1133}=2$, & $a_{1333}=3,$ \\
$a_{2233}=2,$ & $a_{3333}=2$, & $a_{4433}=2$, & $a_{1144}=3$, & $a_{1344}=3$, & $a_{2244}=3,$ \\
$a_{3344}=2,$ & $a_{4444}=2$, & $a_{1212}=2$, & $a_{2312}=2$, & $a_{1113}=3$, & $a_{1313}=2,$ \\
$a_{2213}=2,$ & $a_{3313}=3$, & $a_{4413}=3$, & $a_{1414}=3$, & $a_{3414}=2$, & $a_{1223}=3,$ \\
$a_{2323}=3,$ & $a_{2424}=3$, & $a_{1434}=2$, & $a_{3434}=3.$ \\
\bottomrule
\end{tabular*}
\end{center}
\end{table}

\begin{table}[!htbp]
\begin{center}
\caption{Nonzero components of the symmetric tensor ${\mathcal B}$ for Example \ref{exam1}.}\vskip 0.2mm
\label{tensorB1}
\def\temptablewidth{1\textwidth}
\begin{tabular*}{\temptablewidth}{@{\extracolsep{\fill}}llllll}\toprule
$b_{1111}=1, $ & $b_{1311}=1$,  & $b_{2211}=3$,  & $b_{2311}=1$,  & $b_{2411}=-1$, & $b_{4411}=-2,$ \\
$b_{1122}=1, $ & $b_{1322}=2$,  & $b_{1422}=1$,  & $b_{2222}=2$,  & $b_{2322}=1$,  & $b_{2422}=-1,$ \\
$b_{3422}=1, $ & $b_{4422}=1$,  & $b_{1133}=-2$, & $b_{1233}=1$,  & $b_{1333}=1$,  & $b_{2233}=-2,$ \\
$b_{2433}=-1,$ & $b_{3333}=-1$, & $b_{3433}=1$,  & $b_{4433}=-1$, & $b_{1444}=1,$  & $b_{2244}=-1,$\\
$b_{2444}=-1,$ & $b_{3344}=1$,  & $b_{4444}=-1$, & $b_{1112}=1$,  & $b_{1312}=-1,$ & $b_{1412}=-1,$\\
$b_{2212}=-1,$ & $b_{2312}=1$,  & $b_{3412}=1$,  & $b_{4412}=-1$, & $b_{2213}=-2,$ & $b_{3313}=1, $\\
$b_{3413}=1,$  & $b_{4413}=1$,  & $b_{1214}=2$,  & $b_{1414}=1$,  & $b_{2214}=1,$  & $b_{2314}=1, $ \\
$b_{2414}=2,$  & $b_{1423}=1$,  & $b_{2223}=1$,  & $b_{3323}=2,$  & $b_{4423}=1,$  & $b_{1124}=-2,$ \\
$b_{1224}=-1,$ & $b_{1324}=1$,  & $b_{2324}=1$,  & $b_{3324}=1,$   & $b_{3424}=1,$   & $b_{1134}=-2, $\\
$b_{1234}=1,$  & $b_{1434}=1$,  & $b_{2234}=-2$, & $b_{4434}=-1.$ \\
\bottomrule
\end{tabular*}
\end{center}
\end{table}

\begin{example}\label{exam2}
This example deals with two $4$-th order $6$-dimensional symmetric tensors $\A$ and $\B$, where all components of $\mathcal{B}$ are normally distributed in $(-2,2)$ and then rounded to one digit by utilizing the {\sc Matlab} script `\verb"roundn"'. The second-order cone is given by $\mathcal{K}:=\mathcal{K}^4\times \mathcal{K}^2$, and both tensors $\A$ and $\B$ are specified as listed in Tables \ref{tensorA2} and \ref{tensorB2}, respectively.

\begin{table}[!htbp]
\begin{center}
\caption{Nonzero components of the symmetric tensor ${\mathcal A}$ for Example \ref{exam2}.}\vskip 0.2mm
\label{tensorA2}
\def\temptablewidth{1\textwidth}
\begin{tabular*}{\temptablewidth}{@{\extracolsep{\fill}}cccccc}\toprule
$a_{1111}=2,$ & $a_{1511}=3$, & $a_{2211}=3$, & $a_{3311}=2$, & $a_{4411}=2$, & $a_{5511}=3,$ \\
$a_{6611}=2,$ & $a_{1212}=2$, & $a_{2512}=2$, & $a_{1313}=3$, & $a_{3313}=3$, & $a_{3513}=2,$ \\
$a_{1414}=3,$ & $a_{4514}=2$, & $a_{1115}=2$, & $a_{1515}=3$, & $a_{2215}=3$, & $a_{3315}=3,$ \\
$a_{4415}=3$, & $a_{6615}=3$, & $a_{1616}=2$, & $a_{5616}=3$, & $a_{1122}=2$, & $a_{1522}=2,$ \\
$a_{2222}=2$, & $a_{3322}=2$, & $a_{4422}=2$, & $a_{5522}=3$, & $a_{6622}=2$, & $a_{2323}=3,$ \\
$a_{2424}=3$, & $a_{1225}=2$, & $a_{2525}=3$, & $a_{2626}=2$, & $a_{1133}=2$, & $a_{1533}=2,$ \\
$a_{2233}=2$, & $a_{3333}=2$, & $a_{4433}=3$, & $a_{5533}=2$, & $a_{6633}=3$, & $a_{3434}=2,$ \\
$a_{1335}=3$, & $a_{3535}=2$, & $a_{3636}=2$, & $a_{1144}=3$, & $a_{1544}=3$, & $a_{2244}=3,$ \\
$a_{3344}=3$, & $a_{4444}=2$, & $a_{5544}=2$, & $a_{6644}=2$, & $a_{1445}=3$, & $a_{4545}=2,$ \\
$a_{4646}=2$, & $a_{1155}=2$, & $a_{1555}=2$, & $a_{2255}=2$, & $a_{3355}=3$, & $a_{4455}=3,$ \\
$a_{6655}=2$, & $a_{1656}=2$, & $a_{5656}=3$, & $a_{1166}=3$, & $a_{1566}=2$, & $a_{2266}=2,$ \\
$a_{3366}=2$, & $a_{4466}=2$, & $a_{5566}=2$, & $a_{6666}=2.$ \\
\bottomrule
\end{tabular*}
\end{center}
\end{table}

\end{example}

{\scriptsize\begin{longtable}[!htbp]{llllll}
\caption{Nonzero components of the symmetric tensor ${\mathcal B}$ for Example \ref{exam2}.}\label{tensorB2}\\
\toprule
$b_{1111}=-1.0, $ & $b_{1211}=1.1$,  & $b_{1311}=1.0$,  & $b_{1411}=-0.4$,  & $b_{1511}=0.8$, & $b_{1611}=-0.3,$ \\
$b_{2211}=0.1, $ & $b_{2311}=-0.3$,  & $b_{2411}=-0.8$,  & $b_{2511}=-0.1$,  & $b_{2611}=-0.1$, & $b_{3311}=0.2,$ \\
$b_{3411}=0.9, $ & $b_{3511}=-0.3$,  & $b_{3611}=-1.0$,  & $b_{4411}=1.2$,  & $b_{4511}=0.4$, & $b_{4611}=-0.6,$ \\
$b_{5511}=1.4, $ & $b_{5611}=0.2$,  & $b_{6611}=0.5$,  & $b_{1112}=-2.6$,  & $b_{1212}=1.0$, & $b_{1312}=-0.7,$ \\
$b_{1412}=-1.2, $ & $b_{1512}=1.5$,  & $b_{1612}=0.3$,  & $b_{2212}=0.8$,  & $b_{2312}=0.4$,  & $b_{2412}=-0.4,$ \\
$b_{2612}=0.3, $ & $b_{3312}=-2.3$,  & $b_{3412}=0.8$,  & $b_{3512}=-0.6$,  & $b_{3612}=0.7$,  & $b_{4412}=-0.5,$ \\
$b_{4512}=-2.2, $ & $b_{4612}=0.7$,  & $b_{5512}=-1.6$,  & $b_{5612}=0.2$,  & $b_{1113}=0.1$,  & $b_{1213}=0.3,$ \\
$b_{1313}=-0.1, $ & $b_{1413}=0.3$,  & $b_{1513}=-1.1$,  & $b_{1613}=-0.6$,  & $b_{2213}=-0.3$,  & $b_{2313}=-0.3,$ \\
$b_{2413}=-0.6, $ & $b_{2513}=1.0$,  & $b_{2613}=0.1$,  & $b_{3313}=-1.2$,  & $b_{3413}=-0.3$,  & $b_{3513}=-1.8,$ \\
$b_{3613}=0.3, $ & $b_{4413}=-0.7$,  & $b_{4513}=0.5$,  & $b_{4613}=0.3$,  & $b_{5513}=-0.2$,  & $b_{5613}=1.0,$ \\
$b_{6613}=0.9, $ & $b_{1114}=1.0$,  & $b_{1214}=-0.8$,  & $b_{1414}=-0.1$,  & $b_{1514}=-0.4$,  & $b_{1614}=0.4,$ \\
$b_{2214}=1.7, $ & $b_{2314}=1.0$,  & $b_{2414}=0.6$,  & $b_{2514}=0.2$,  & $b_{2614}=1.4$,  & $b_{3314}=0.4,$ \\
$b_{3414}=-0.9, $ & $b_{3514}=-0.6$,  & $b_{3614}=0.3$,  & $b_{4414}=-0.3$,  & $b_{4514}=0.2$,  & $b_{4614}=0.3,$ \\
$b_{5514}=-1.0, $ & $b_{5614}=0.2$,  & $b_{6614}=-0.8$,  & $b_{1115}=0.1$,  & $b_{1215}=0.3$,  & $b_{1315}=0.3,$ \\
$b_{1415}=-0.1, $ & $b_{1615}=-0.2$,  & $b_{2215}=-0.5$,  & $b_{2315}=-0.7$,  & $b_{2415}=0.7$,  & $b_{2515}=-0.1,$ \\
$b_{2615}=-0.2, $ & $b_{3315}=0.4$,  & $b_{3415}=-0.2$,  & $b_{3515}=1.8$,  & $b_{3615}=0.4$,  & $b_{4415}=-1.1,$ \\
$b_{4515}=0.4, $ & $b_{4615}=-0.1$,  & $b_{5515}=-1.2$,  & $b_{5615}=-0.4$,  & $b_{6615}=-1.2$,  & $b_{1116}=1.2,$ \\
$b_{1216}=1.6, $ & $b_{1316}=0.4$,  & $b_{1416}=-0.3$,  & $b_{1516}=1.4$,  & $b_{1616}=0.2$,  & $b_{2216}=0.2,$ \\
$b_{2316}=0.8, $ & $b_{2416}=-1.1$,  & $b_{2516}=-0.4$,  & $b_{2616}=0.1$,  & $b_{3316}=-0.1$,  & $b_{3416}=-0.6,$ \\
$b_{3516}=0.1, $ & $b_{3616}=-0.6$,  & $b_{4416}=-0.7$,  & $b_{4516}=-0.7$,  & $b_{4616}=0.8$,  & $b_{5516}=0.8,$ \\
$b_{5616}=-0.5, $ & $b_{6616}=-1.2$,  & $b_{1122}=0.6$,  & $b_{1222}=-1.2$,  & $b_{1322}=-1.4$,  & $b_{1422}=-1.2,$ \\
$b_{1522}=0.3, $ & $b_{1622}=-0.4$,  & $b_{2222}=-0.4$,  & $b_{2322}=0.5$,  & $b_{2422}=1.2$,  & $b_{2522}=1.1,$ \\
$b_{2622}=-0.8, $ & $b_{3322}=0.9$,  & $b_{3422}=0.4$,  & $b_{3522}=1.0$,  & $b_{3622}=-0.3$,  & $b_{4422}=0.3,$ \\
$b_{4522}=-0.1, $ & $b_{4622}=0.6$,  & $b_{5522}=0.3$,  & $b_{5622}=1.0$,  & $b_{6622}=0.1$,  & $b_{1123}=0.9,$ \\
$b_{1223}=0.3, $ & $b_{1323}=1.4$,  & $b_{1423}=-0.8$,  & $b_{1523}=-0.7$,  & $b_{1623}=0.5$,  & $b_{2223}=1.0,$ \\
$b_{2323}=-0.5, $ & $b_{2423}=-0.9$,  & $b_{2523}=-0.1$,  & $b_{2623}=-1.4$,  & $b_{3323}=-0.6$,  & $b_{3523}=0.1,$ \\
$b_{3623}=0.4, $ & $b_{4423}=1.5$,  & $b_{4523}=0.8$,  & $b_{4623}=0.1$,  & $b_{5523}=0.2$,  & $b_{5623}=1.7,$ \\
$b_{6623}=0.6, $ & $b_{1124}=0.2$,  & $b_{1224}=-0.2$,  & $b_{1324}=-0.6$,  & $b_{1424}=-1.2$,  & $b_{1524}=-0.7,$ \\
$b_{1624}=0.8, $ & $b_{2224}=-0.4$,  & $b_{2324}=0.5$,  & $b_{2424}=-0.1$,  & $b_{2524}=-0.5$,  & $b_{2624}=0.6,$ \\
$b_{3324}=2.3, $ & $b_{3424}=0.4$,  & $b_{3524}=0.2$,  & $b_{3624}=-1.1$,  & $b_{4424}=-1.1$,  & $b_{4524}=0.9,$ \\
$b_{4624}=-0.9, $ & $b_{5624}=-0.9$,  & $b_{6624}=-0.6$,  & $b_{1125}=0.7$,  & $b_{1225}=0.5$,  & $b_{1325}=-0.8,$ \\
$b_{1425}=1.2, $ & $b_{1525}=0.3$,  & $b_{1625}=-0.4$,  & $b_{2225}=-0.9$,  & $b_{2425}=0.1$,  & $b_{2525}=-0.1,$ \\
$b_{2625}=1.1, $ & $b_{3325}=-2.1$,  & $b_{3425}=-0.8$,  & $b_{3525}=0.5$,  & $b_{3625}=-0.4$,  & $b_{4425}=-0.1,$ \\
$b_{4525}=0.5, $ & $b_{4625}=0.7$,  & $b_{5525}=0.3$,  & $b_{5625}=-0.7$,  & $b_{6625}=-1.6$,  & $b_{1126}=-2.0$ \\
$b_{1226}=1.1, $ & $b_{1326}=0.6$,  & $b_{1426}=1.3$,  & $b_{1526}=0.6$,  & $b_{1626}=-0.5$,  & $b_{2226}=1.0,$ \\
$b_{2326}=0.5, $ & $b_{2426}=0.4$,  & $b_{2526}=-0.2$,  & $b_{2626}=0.4$,  & $b_{3326}=1.7$,  & $b_{3426}=0.7,$ \\
$b_{3526}=-0.4, $ & $b_{3626}=0.1$,  & $b_{4426}=-0.4$,  & $b_{4526}=-0.7$,  & $b_{4626}=0.2$,  & $b_{5526}=0.2,$ \\
$b_{5626}=-0.7, $ & $b_{6626}=0.5$,  & $b_{1133}=-0.5$,  & $b_{1233}=0.8$,  & $b_{1333}=0.9$,  & $b_{1433}=-0.3,$ \\
$b_{1533}=0.1, $ & $b_{1633}=0.6$,  & $b_{2233}=-1.1$,  & $b_{2333}=0.3$,  & $b_{2433}=-0.6$,  & $b_{2533}=-1.1,$ \\
$b_{2633}=-1.2, $ & $b_{3333}=0.8$,  & $b_{3433}=-0.6$,  & $b_{3533}=1.5$,  & $b_{3633}=0.8$,  & $b_{4433}=-1.5,$ \\
$b_{4533}=-0.3, $ & $b_{4633}=0.7$,  & $b_{5533}=-0.1$,  & $b_{5633}=0.2$,  & $b_{6633}=1.6$,  & $b_{1134}=-0.9,$ \\
$b_{1234}=0.7, $ & $b_{1334}=-0.4$,  & $b_{1434}=-0.6$,  & $b_{1534}=0.2$,  & $b_{1634}=0.4$,  & $b_{2234}=-0.2,$ \\
$b_{2334}=-0.4, $ & $b_{2434}=-1.0$,  & $b_{2534}=-0.3$,  & $b_{2634}=1.6$,  & $b_{3334}=1.0$,  & $b_{3434}=0.1,$ \\
$b_{3534}=-0.6, $ & $b_{3634}=1.3$,  & $b_{4434}=1.0$,  & $b_{4534}=0.9$,  & $b_{4634}=-1.0$,  & $b_{5534}=-1.3,$ \\
$b_{5634}=1.1, $ & $b_{6634}=2.3$,  & $b_{1135}=-2.3$,  & $b_{1235}=-0.1$,  & $b_{1335}=-0.2$,  & $b_{1435}=-0.1,$ \\
$b_{1535}=0.2, $ & $b_{1635}=-0.3$,  & $b_{2235}=-0.3$,  & $b_{2335}=-0.3$,  & $b_{2435}=-0.1$,  & $b_{2535}=-0.3,$ \\
$b_{2635}=-0.9, $ & $b_{3335}=-0.4$,  & $b_{3435}=0.1$,  & $b_{3535}=0.4$,  & $b_{4435}=-0.4$,  & $b_{4535}=0.1,$ \\
$b_{4635}=-0.2, $ & $b_{5535}=1.6$,  & $b_{5635}=-0.9$,  & $b_{6635}=-0.4$,  & $b_{1136}=1.1$,  & $b_{1236}=0.5,$ \\
$b_{1336}=0.2, $ & $b_{1436}=-0.5$,  & $b_{1536}=-0.3$,  & $b_{1636}=-0.1$,  & $b_{2236}=0.2$,  & $b_{2336}=-0.2,$ \\
$b_{2436}=-0.5, $ & $b_{2536}=-0.6$,  & $b_{2636}=0.1$,  & $b_{3336}=-0.6$,  & $b_{3436}=1.4$,  & $b_{3536}=0.6,$ \\
$b_{3636}=0.1, $ & $b_{4436}=-0.7$,  & $b_{4536}=0.7$,  & $b_{4636}=0.9$,  & $b_{5536}=1.0$,  & $b_{5636}=0.2,$ \\
$b_{6636}=-1.0, $ & $b_{1144}=-1.4$,  & $b_{1244}=-0.8$,  & $b_{1344}=0.6$,  & $b_{1444}=-2.2$,  & $b_{1544}=-0.3,$ \\
$b_{1644}=0.9, $ & $b_{2244}=-1.1$,  & $b_{2344}=0.6$,  & $b_{2444}=-0.5$,  & $b_{2544}=0.2$,  & $b_{2644}=-0.7,$ \\
$b_{3344}=1.6, $ & $b_{3444}=-0.8$,  & $b_{3544}=0.1$,  & $b_{3644}=-0.5$,  & $b_{4444}=-0.8$,  & $b_{4544}=0.3,$ \\
$b_{4644}=0.4, $ & $b_{5544}=-0.3$,  & $b_{5644}=0.4$,  & $b_{6644}=1.1$,  & $b_{1145}=-1.3$,  & $b_{1245}=-0.6,$ \\
$b_{1445}=-1.3, $ & $b_{1545}=-0.2$,  & $b_{1645}=0.2$,  & $b_{2345}=0.2$,  & $b_{2445}=0.1$,  & $b_{2545}=-1.5,$ \\
$b_{2645}=0.4, $ & $b_{3345}=0.2$,  & $b_{3445}=-0.8$,  & $b_{3545}=0.9$,  & $b_{3645}=-0.3$,  & $b_{4445}=0.5,$ \\
$b_{4545}=1.6, $ & $b_{4645}=1.6$,  & $b_{5545}=1.5$,  & $b_{5645}=-0.9$,  & $b_{6645}=-0.6$,  & $b_{1246}=1.1,$ \\
$b_{1346}=0.1, $ & $b_{1446}=0.6$,  & $b_{1546}=0.7$,  & $b_{1646}=0.5$,  & $b_{2246}=0.4$,  & $b_{2346}=0.3,$ \\
$b_{2446}=-0.3, $ & $b_{2546}=-0.3$,  & $b_{2646}=0.4$,  & $b_{3346}=0.6$,  & $b_{3446}=0.4$,  & $b_{3546}=-1.0,$ \\
$b_{3646}=-0.1, $ & $b_{4446}=0.5$,  & $b_{4546}=0.2$,  & $b_{4646}=0.9$,  & $b_{5546}=0.8$,  & $b_{5646}=0.5,$ \\
$b_{6646}=-1.8, $ & $b_{1155}=0.4$,  & $b_{1255}=2.2$,  & $b_{1355}=1.1$,  & $b_{1455}=1.0$,  & $b_{1555}=-0.1,$ \\
$b_{1655}=0.9, $ & $b_{2255}=-1.1$,  & $b_{2355}=0.1$,  & $b_{2455}=0.5$,  & $b_{2555}=0.2$,  & $b_{2655}=-0.7,$ \\
$b_{3355}=-0.7, $ & $b_{3455}=0.4$,  & $b_{3555}=0.9$,  & $b_{3655}=0.3$,  & $b_{4455}=-0.9$,  & $b_{4555}=-0.2,$ \\
$b_{4655}=-1.0, $ & $b_{5555}=-0.2$,  & $b_{6655}=-0.4$,  & $b_{1156}=-0.5$,  & $b_{1256}=-0.5$,  & $b_{1356}=-0.7,$ \\
$b_{1456}=0.4, $ & $b_{1556}=0.1$,  & $b_{1656}=0.1$,  & $b_{2256}=0.3$,  & $b_{2356}=0.8$,  & $b_{2456}=0.1,$ \\
$b_{2556}=-0.1, $ & $b_{2656}=0.7$,  & $b_{3356}=-0.8$,  & $b_{3456}=-0.3$,  & $b_{3556}=0.5$,  & $b_{3656}=-0.8,$ \\
$b_{4456}=-0.5, $ & $b_{4656}=0.9$,  & $b_{5556}=-0.2$,  & $b_{5656}=0.2$,  & $b_{6656}=-0.6$,  & $b_{1166}=1.5,$ \\
$b_{1266}=-0.6, $ & $b_{1366}=0.6$,  & $b_{1466}=-1.3$,  & $b_{1566}=1.2$,  & $b_{1666}=0.5$,  & $b_{2266}=-0.4,$ \\
$b_{2466}=0.9, $ & $b_{2566}=-0.2$,  & $b_{2666}=0.3$,  & $b_{3366}=-0.2$,  & $b_{3466}=0.5$,  & $b_{3566}=-1.7,$ \\
$b_{3666}=-0.6, $ & $b_{4466}=1.8$,  & $b_{4566}=-1.1$,  & $b_{4666}=0.6$,  & $b_{5566}=-0.4$,  & $b_{5666}=0.4,$ \\
$b_{6666}=-1.0.$ &&&&&\\
\bottomrule
\end{longtable}}

Following the suggestion in \cite{LHQ15}, we use
\[\label{error}{\rm RelErr:}=\|y^{(k)}\|:=\|\B (x^{(k)})^{m-1}-\lambda_k\A (x^{(k)})^{m-1}\|\leq {\rm Tol}\]
to be the termination criterion and attain an approximate numerical solution with a preset tolerance `${\rm Tol}$'. Now, we test four scenarios of `${\rm Tol}$' by setting ${\rm Tol}:=\left\{10^{-3}\right.$, $5\cdot 10^{-4}$, $10^{-4}$, $\left.5\cdot 10^{-5}\right\}$. In our experiments, we consider two cases of the starting point $u^{(0)}\in\mathcal{K}$, which is generated in two steps. The first step is that we generate two vectors $z^{(0)}$: one is a vector of ones, i.e., $z^{(0)}=(1,\cdots,1)^\t$, the other one is a random vector uniformly distributed in $(0,1)$. Then, to guarantee the starting point $u^{(0)}\in\mathcal{K}$, we project the intermediate vectors $z^{(0)}$ onto the second-order cone $\mathcal{K}$, i.e., $u^{(0)}=\Pi_{\mathcal{K}}(z^{(0)})$ in the second step. As suggested in \cite{LHQ15}, we throughout the experiments take $\alpha$ as $\alpha=5$. To support that the SPA is reliable for finding one of solutions of SOCTEiCPs, we report the number of iterations (`Iter.'), computing time in seconds (`Time'), the relative error (`RelErr') defined by \eqref{error}, eigenvalue (`EigValue') and the associated eigenvector (`EigVector'). The numerical results with respect to different initial points are listed in Tables \ref{table1} and \ref{table2}, respectively.

\setlength\rotFPtop{0pt plus 1fil}
\begin{sidewaystable}
\begin{center}
\caption{Computational results for the case where the intermediate vector $z^{(0)}=(1,\cdots,1)^\t$.}\vskip 0.2mm
\label{table1}
\def\temptablewidth{1\textwidth}
\begin{tabular*}{\temptablewidth}{@{\extracolsep{\fill}}lcccccc}\toprule
Example & Tol & Iter. & Time & RelErr & EigValue & EigVector\\\midrule
Example \ref{exam1}& 1.0e-03 & 145 & 0.05 & 9.965e-04 & 0.1618 & $(0.1222,0.0391,0.5429,0.2702)^\t$ \\
Example \ref{exam2}& 1.0e-03 & 202 & 0.07 & 9.957e-04 & 0.1664 & $(0.3517,0.2774,-0.0248,-0.1473,0.2793,-0.0696)^\t$ \\
\midrule
Example \ref{exam1}& 5.0e-04 & 276 & 0.07 & 4.995e-04 & 0.1616 & $(0.1221,0.0390,0.5432,0.2700)^\t$ \\
Example \ref{exam2}& 5.0e-04 & 578 & 0.19 & 5.000e-04 & 0.1664 & $(0.3519,0.2777,-0.0251,-0.1472,0.2788,-0.0698)^\t$ \\
\midrule
Example \ref{exam1}& 1.0e-04 & 1193 & 0.31 & 9.992e-05 & 0.1614 & $(0.1221,0.0388,0.5433,0.2699)^\t$ \\
Example \ref{exam2}& 1.0e-04 & 3705 & 1.28 & 9.997e-05 & 0.1665 & $(0.3518,0.2775,-0.0258,-0.1480,0.2785,-0.0700)^\t$ \\
\midrule
Example \ref{exam1}& 5.0e-05 & 2270 & 0.57 & 4.999e-05 & 0.1613 & $(0.1221,0.0388,0.5433,0.2699)^\t$ \\
Example \ref{exam2}& 5.0e-05 & 6100 & 2.02 & 5.000e-05 & 0.1665 & $(0.3518,0.2775,-0.0258,-0.1481,0.2785,-0.0700)^\t$ \\
\bottomrule
\end{tabular*}
\end{center}
\end{sidewaystable}

\setlength\rotFPtop{0pt plus 1fil}
\begin{sidewaystable}
\begin{center}
\caption{Computational results for the case where $z^{(0)}$ is a random intermediate vector.}\vskip 0.2mm
\label{table2}
\def\temptablewidth{1\textwidth}
\begin{tabular*}{\temptablewidth}{@{\extracolsep{\fill}}lcccccc}\toprule
Example & Tol & Iter. & Time & RelErr & EigValue & EigVector\\\midrule
Example \ref{exam1}& 1.0e-03 & 117 & 0.04 & 9.918e-04 & 0.1618 & $(0.1220,0.0393,0.5434,0.2699)^\t$ \\
Example \ref{exam2}& 1.0e-03 & 455 & 0.16 & 9.991e-04 & 0.1667 & $(0.3514,0.2770,-0.0264,-0.1492,0.2786,-0.0701)^\t$ \\
\midrule
Example \ref{exam1}& 5.0e-04 & 309 & 0.08 & 4.991e-04 & 0.1611 & $(0.1220,0.0387,0.5436,0.2697)^\t$ \\
Example \ref{exam2}& 5.0e-04 & 816 & 0.28 & 4.993e-04 & 0.1665 & $(0.3519,0.2777,-0.0257,-0.1477,0.2785,-0.0700)^\t$ \\
\midrule
Example \ref{exam1}& 1.0e-04 & 1052 & 0.27 & 9.994e-05 & 0.1613 & $(0.1221,0.0388,0.5433,0.2700)^\t$ \\
Example \ref{exam2}& 1.0e-04 & 3435 & 1.18 & 9.996e-05 & 0.1665 & $(0.3518,0.2775,-0.0258,-0.1480,0.2785,-0.0700)^\t$ \\
\midrule
Example \ref{exam1}& 5.0e-05 & 2837 & 0.71 & 4.999e-05 & 0.1613 & $(0.1221,0.0388,0.5433,0.2700)^\t$ \\
Example \ref{exam2}& 5.0e-05 & 9151 & 3.04 & 4.999e-05 & 0.1665 & $(0.3517,0.2774,-0.0258,-0.1481,0.2786,-0.0700)^\t$ \\  \bottomrule
\end{tabular*}
\end{center}
\end{sidewaystable}

It can be easily seen from the data reported in Tables \ref{table1} and \ref{table2} that the SPA can successfully find some eigenvector-eigenvalue pairs of SOCTEiCPs, even though it seems that the number of iterations increases significantly as the precision improvement on solutions. All numerical results sufficiently show that the SPA is a reliable solver for SOCTEiCPs.

\section{Conclusions}\label{Concl}
We study a class of SOCTEiCPs, which generalize the SOCEiCP for matrices introduced in recent paper \cite{AR14,FFJS15}. Although SOCTEiCP \eqref{SOCTEiCP} is a specific case of TGEiCP, we propose a new potentially helpful variational inequality reformulation for the problem under consideration. As we know, variational inequality is a powerful tool for mathematics analysis. Thus, such a variational inequality characterization might help us analyze further properties of SOCTEiCP \eqref{SOCTEiCP}, which are also our future concerns. Besides, we consider a special case of SOCTEiCP \eqref{SOCTEiCP} with two symmetric tensors $\mathcal{A}$ and $\mathcal{B}$, and present a nonlinear programming reformulation. To break through the limitation of the symmetry condition, we discuss a class of slightly general SOCTEiCPs with sub-symmetric tensors, and similarly show that solving the sub-symmetric SOCTEiCP reduces to finding a stationary point of a nonlinear programming problem. However, our results do not completely break the bottleneck of (sub-) symmetry condition, in the future, we will study more general SOCTEiCPs in absence of symmetric and sub-symmetric properties.

\begin{acknowledgements}
The authors would like to thank the two anonymous referees for their careful reading and valuable comments, which help us improve the presentation of this paper greatly.
This work was supported by National Natural Science Foundation of China (NSFC) at Grant Nos. (11171083, 11301123, 11571087) and Natural Science Foundation of Zhejiang Province at Grant No. LZ14A010003.
\end{acknowledgements}


\end{document}